\documentclass[12pt]{amsart}
\usepackage{amsmath,amsthm,amsfonts,amssymb,latexsym}

\usepackage{hyperref}

\numberwithin{equation}{section}  %%% equations numbers (A.B)

\theoremstyle{remark}

\theoremstyle{plain}

\newtheorem{lem}{Lemma}[section]
\newtheorem{thm}{Theorem}

\numberwithin{equation}{section}

\newcommand{\ZZ}{{\mathbb Z}}

\newcommand{\NN}{{\mathbb N}}

\newcommand{\cD}{\ensuremath{\mathcal{D}}}
\newcommand{\cE}{\ensuremath{\mathcal{E}}}

\newcommand{\cV}{\ensuremath{\mathcal{V}}}

\DeclareMathOperator{\lcm}{lcm}

\makeatletter
\renewcommand{\pmod}[1]{\allowbreak\mkern7mu({\operator@font mod}\,\,#1)}
\makeatother

\newcommand{\be}{\begin{equation}}
\newcommand{\ee}{\end{equation}}

\renewcommand{\a}{\ensuremath{\alpha}}

\renewcommand{\le}{\leqslant}

\renewcommand{\ge}{\geqslant}

\newcommand{\DHL}{\text{DHL}}
\begin{document}

\title{Gaps between totients}
\author{Kevin Ford}
\address{Department of Mathematics, 1409 West Green Street, University
of Illinois at Urbana-Champaign, Urbana, IL 61801, USA}
\email{ford@math.uiuc.edu}
\author{Sergei Konyagin}
\address{Steklov Institute of Mathematics, 8 Gubkin Street, Moscow,
119991, Russia} 
\email{konyagin@mi-ras.ru}

\begin{abstract}
We study the set $\cD$ of positive integers $d$ 
for which the equation $\phi(a)-\phi(b)=d$ has infinitely
many solution pairs $(a,b)$.  We show that $\min \cD \le 154$,
exhibit a specific $A$ so that every multiple of $A$ is in $\cD$,
and show that any progression $a\mod d$ with $4|a$ and $4|d$,
contains infinitely many elements of $\cD$.
We also show that the 
Generalized Elliott-Halberstam Conjecture,
as defined in \cite{polymath}, implies that $\cD$
contains all positive, even integers.
\end{abstract}

\date{\today}
\thanks{2010 MSC 11A25, 11N64 (primary), 11D85 (secondary). }
\thanks{First author supported by National Science Foundation Grant
DMS-1802139.}
\maketitle

%%%%%%%%%%%%%%%%%%%%%%%%%

\section{Introduction}

Let $\cV=\{v_1,v_2,\ldots\}$ be the set of totients, that is,
$\cV$ is the image of Euler's totient function $\phi(n)$.
In this paper we study the set $\cD$ of positive integers
which are infinitely often a difference of two elements of
$\cV$.  
A classical conjecture asserts that every even positive integer
is infinitely often the difference of two primes, and this implies
immediately that $\cD$ is the set of all positive, even integers.
We are interested in what can be accomplished unconditionally,
by leveraging the recent breakthroughs on gaps between consecutive primes by Zhang \cite{Zhang}, Maynard \cite{Maynard}, Tao (unpublished) and
the PolyMath8b project \cite{polymath}.
We let $\cE$ be the set of positive even numbers that
are infinitely often the difference of two primes.
Clearly $\cE \subseteq \cD$.
In this note we prove some results about $\cD$ which are
not known for $\cE$.

The behavior of the smallest elements of $\cD$ 
arose in recent work of Fouvry and Waldschmidt \cite{FW}
concerning representation of integers by cyclotomic forms,
%SK and the problem of studying $\cD$ was also posted
and the problem of studying the differences of totients was also posted
in a list of open problems by Shparlinski \cite[Problem 56]{shparlinski}.
Our paper is a companion of the recent work of the
first author \cite{Ford} concerning the equation
$\phi(n+k)=\phi(n)$ for fixed $k$.

It is known \cite{polymath} that $\min \cE \le 246$ and thus 
$\min \cD\le 246$.
We can do somewhat better.

\begin{thm}\label{uncond}
We have  $\min \cD\le 154$.
\end{thm}

Although there is no specific even integer which is known
to be infinitely often the difference of two primes,
we give an infinite family of specific numbers that
are in $\cD$.

\begin{thm}\label{specific}
Let $a_0=\prod_{p\le 47} p$ and $b=\lcm[1,2,\ldots,49]$.
Then every multiple of $\phi(a_0 b_0)a_0$
lies in $\cD$.
\end{thm}

Granville, Kane, Koukoulopoulos and Lemke-Oliver \cite{GKKL} showed that
$\cE$ has lower asymptotic density at least
$\frac{1}{354}$ and thus so does $\cD$.
We do not know how to prove a better lower bound
for the density
and leave this as an open problem.

\medskip

Central to the works \cite{Maynard, Maynard-dense, polymath, Zhang}
is the concept of an admissible set of linear forms.
For positive integers $a_i$ and integers $b_i$, the set of affine-linear
forms $a_1x+b_1,\ldots,a_kx+b_k$ is admissible if, for every
prime $p$, there is an $x\in \ZZ$ such that
$p \nmid (a_1x+b_1)\cdots (a_kx+b_k)$.

\textbf{Definition}.  Hypothesis $\DHL[k,m]$ is the statement
that for any admissible $k$-tuple of linear forms
$a_in+b_i$, $1\le i\le k$, for infinitely many $n$,
at least $m$ of them are simultaneously prime.

In this paper we are concerned with the statements $\DHL[k,2]$.
The Polymath8b project \cite{polymath}, plus subsequent work of Maynard \cite{Maynard-dense},
established $\DHL[50,2]$ unconditionally.

The Elliott-Halberstam Conjecture implies $\DHL[5,2]$, see
\cite{Maynard}.
The Generalized Elliott-Halberstam Conjecture implies
$\DHL[3,2]$ (see \cite{polymath} for details).

%\newpage

\begin{thm}\label{DHL}
We have
\begin{enumerate}
\item[(i)] $\DHL[3,2]$ implies that $\cD=\{2,4,6,8,10,\ldots\}$,
the set of all positive even integers.
\item[(ii)] $\DHL[4,2]$ implies that $\cD$
contains every positive multiple of 4.
\item[(iii)] $\DHL[5,2]$ implies that $\min \cD \le 6$.
\item[(iv)] $\DHL[6,2]$ implies $8\in \cD$.
\end{enumerate}
\end{thm}

By contrast, for any $k\ge 2$, $\DHL[k,2]$ implies that 
$\liminf p_{n+1} - p_n \le a_k$, where 
$a_k$ is the minimum of $h_{k}-h_1$ over all
admissible $k-$tuples $n+h_1,\ldots,n+h_k$.
We have $a_3=6$, $a_4=8$, $a_5=12$ 
and $a_6=16$.

%%KF
We show parts of Theorem \ref{DHL} (ii) and (iii) using
a more general result.

\begin{thm}\label{DHLk}
Assume $\DHL[k,2]$, with $k\ge 3$.   Also assume that there are
integers $1<m_1<\ldots<m_k$ and $\ell_{i,j}$ for
$1\le i<j\le k$ and such that
\[
\frac{\ell_{i,j} m_i}{m_j-m_i}\in \cV, \;\; 
\frac{\ell_{i,j} m_j}{m_j-m_i} \in \cV \qquad (1\le i< j\le k).
\]
Then $L\le 2\max \ell_{i,j}$.
Moreover, if $\ell_{i,j}=\ell$ for all $i,j$
then $2\ell \in \cD$.
\end{thm}

In Section \ref{sec:heuristic}, we give a 
heuristic argument that 
there exist numbers $m_1,\ldots,m_{50}$
satisfying the hypothesis of Theorem \ref{DHLk} with 
$\ell_{i,j}=2$ for all $i<j$.
In this case we achieve an unconditional proof
that $4\in \cD$.  Actually finding such $m_i$ seems
computationally difficult, however.

Does every arithmetic progression $a\bmod d$ containing even numbers
have infinitely many elements of $\cD$?  We answer in 
the affirmative if $4|d$ and $4|a$.  The case
$a\equiv 2\pmod{4}$ is more difficult; see our
Remarks following the proof in Section \ref{sec:ap}.

\begin{thm}\label{ap}
Let $a,d$ be positive integers with $4|a$, $4|d$.  Then 
the progressions $a\mod d$ contains infinitely many
elements of $\cD$.
\end{thm}

%SK
Observe that, even assuming  $\DHL[3,2]$, there is no specific progression 
$a\mod d$, not containing $0$, which is known to contain a number that is  
infinitely often the difference of two primes.

When $4|d$ and $a\equiv 2\pmod{4}$ we can sometimes show that
$\cD$ contains infinitely many elements that are $\equiv a\pmod{d}$;
see the Remarks following the proof of Theorem \ref{ap}
in Section \ref{sec:ap}.

\section{Proof of Theorems \ref{uncond}--\ref{DHLk}}

\begin{proof}[Proof of Theorem \ref{uncond}]
Define
\begin{align*}
S_1 &= \{41,43,47,53,67,71\}, \\
S_2 &= \{59,61,67,71,73,83,89,101,103,107,109,113,127,131,137,139\},
\\
S_4 &= \{p\text{ prime}: 127 \le p\le 271\}.
\end{align*}
and consider the collection of 50 linear forms
\[
n+a \;\; (a\in S_1), \quad 2n+a \;\; (a\in S_2), \quad
4n+a \;\; (n\in S_4).
\]
This collection is admissible; indeed if $p<41$ and $n=0$
then all of them are coprime to $p$.  For $p>50$
it is clear that there is an $n$ for which all of them
are coprime to $p$.  When $p\in \{41,43,47\}$
we take $n=1,3,8$ ,respectively, and then all of the forms
are coprime to $p$.
By $\DHL[50,2]$, there are two of these forms that are
simultaneously prime for
infinitely many $n$.
If both forms are of the type $n+a$ for $a\in S_1$,
then this shows that $\min \cD \le 71-41= 30$.
Likewise, if both forms are of the type $2n+a$
for $a\in S_2$ then $\min \cD \le 139-59=80$
and if both forms are of the type $4n+a$ where $a\in S_4$,
then $\min \cD \le 271-127=144$.
 Now suppose for infinitely many
$n$, $n+a$ and $2n+b$ are both prime, where $a\in S_1$, $b\in S_2$.
Then
\[
\phi(4(n+a))=2n+2a-2, \quad \phi(2n+b)=2n+b-1,
\]
which shows that $|b-2a+1|\in \cD$.
We have $b-2a+1\ne 0$ for all choices, and the maximum
of $|b-2a+1|$ is 82, and hence $\min \cD \le 82$.
Similarly, if for  infinitely many
$n$, $2n+a$ and $4n+b$ are both prime, where $a\in S_2$, $b\in S_4$,
then $|b-2a+1|\in \cD$.  Hence $\min \cD \le 154$.
Finally, if  for  infinitely many
$n$, $n+a$ and $4n+b$ are both prime, where $a\in S_1$, $b\in S_4$,
then $|b-4a+3|\in \cD$ since
\[
\phi(8(n+a))=4(n+a-1), \qquad \phi(4n+b)=4n+b-1.
\] 
In all cases $0<|b-4a+3| \le 154$.
\end{proof}

\begin{proof}[Proof of Theorem \ref{specific}]
Let 
$$a_0=\prod_{p\le47}p,
\quad b_0=\lcm[1,2,\ldots,49].$$
Let $k\in \NN$ and consider the admissible set of 
linear forms $n+ka_0, n+2ka_0, \ldots, n+50ka_0$.
Since $\DHL[50,2]$ holds,
for any $k\in\NN$ there exists $j_k\in\{1,\dots,49\}$
such that the equation
$$\phi(u)-\phi(v)=u-v=kj_k a_0$$
has infinitely many solutions in primes $u,v$.
Since $a_0\mid a_0b_0/j_k$, $\phi(a_0b_0 l/j_k) = \phi(a_0b_0 l)/j_k$ for any $l\in\NN$.
Therefore, we have 
$$\phi(a_0b_0 u/j_k) - \phi(a_0b_0 v/j_k) 
=\phi(a_0b_0)(\phi(u)-\phi(v))/j_k = \phi(a_0b_0)a_0 k,$$
as required.
\end{proof}

\begin{proof}[Proof of Theorem \ref{DHLk}]
The set of forms $m_1n-1,\ldots,m_kn-1$ is clearly admissible.
By $\DHL[k,2]$, for some pair $i<j$ and for infinitely many $n$ ,
$m_in-1$ and $m_jn-1$ are prime.
Let $\ell=\ell_{i,j}$ and suppose
that  $x,y$ satisfy 
\[
\phi(x) = \frac{\ell m_i}{m_j-m_i}, \qquad
\phi(y) = \frac{\ell m_j}{m_j-m_i}.
\]
%SK  Then
Then for sufficiently large $n$
\[
\phi( x (m_jn-1) ) - \phi (y (m_in-1)) = 
\frac{-2\ell m_i+2\ell m_j}{m_j-m_i} = 2\ell. \qedhere
\]
\end{proof}

\begin{proof}[Proof of Theorem \ref{DHL}]
(i)  Let $h\in \NN$ and
consider the triple $\{n+1,n+2h+1,2n+2h+1\}$. 
This is admissible, since when $n=0$, all
 of the forms are odd, and similarly
 none are divisible by 3 for some $n\in \{0,1\}$.
By $\DHL[3,2]$, either (i) $n+1$ and $n+2h+1$ 
are infinitely often both prime, (ii) $n+1$
and $2n+2h+1$ are  infinitely often both prime
or (iii) $n+2h+1$
and $2n+2h+1$ are  infinitely often both prime.
In case (i) we have $\phi(n+2h+1)-\phi(n+1)=2h$,
in case (ii) we have $\phi(2n+2h+1)-\phi(4(n+1))=2h$,
and in case (iii) we have $\phi(4(n+2h+1))-\phi(2n+2h+1)=2h$. 

(ii) The deduction $4\in \cD$ follows from Theorem \ref{DHLk} 
using $\ell_{i,j}=2$ for all $i,j$ and
\[\{m_1,\ldots,m_4\} = \{6,8,9,12\}.
\]
Now suppose that $d$ is congruent to 0 or 4 modulo 12,
and define $a$ by $d=2(a+1)$.  In particular, $(a,6)=1$.
Thus, the set of forms $m_in-a$, $1\le i\le 4$, are admissible.
%SK
By $\DHL[4,2]$, for some pair $i<j$ and for infinitely many $n$ ,
$m_in-a$ and $m_jn-a$ are prime. Suppose
that  $x,y$ satisfy 
\[
\phi(x) = \frac{2 m_i}{m_j-m_i}, \qquad
\phi(y) = \frac{2 m_j}{m_j-m_i}.
\]
Then for sufficiently large $n$
\[
\phi( x (m_jn-a) ) - \phi (y (m_in-a)) = 
%\frac{-2\ell m_i+2\ell m_j}{m_j-m_i} = 
2(a+1) = d. \qedhere
\]
Hence, $d\in \cD$.

%SK and the proof of Theorem \ref{DHLk} implies that
%SK $d=2(a+1)\in \cD$.
Finally, if $d\equiv 8\pmod{12}$, write $d=2(b-1)$,
 so that $(b,6)=1$.  Similarly, the set of forms
 $m_i n + b$, $1\le i\le 4$, are admissible
 and we conclude that $d\in \cD$.

(iii) Consider the admissible set of forms
$\{f_1(n),\ldots,f_5(n)\}=\{n, n+2, 2n+1, 4n-1, 4n+3\}$.
Indeed, if $n\equiv 11\pmod{30}$ then all
of the forms are coprime to $30$.
By $\DHL[5,2]$, for some $i<j$ and infinitely
many $n$, $f_i(n)$ and $f_j(n)$ are both prime.
Say $f_i(n)=an+b$ and $f_j(n)=cn+d$
with $c/a \in \{1,2,4\}$.  
Then \[
\phi((2c/a)(an+b))=(c/a)(an+b-1)=cn+(c/a)(b-1)
\]
and
\[
\phi(cn+d) = cn+d-1.
\]
Thus, $|(c/a)(b-1)-(d-1)| \in \cD$.
In all cases, $|(c/a)(b-1)-(d-1)|\le 6$.

(iv) Use Theorem \ref{DHLk} with the set
\[
\{h-72, h-66, h-64, h-63, h-60, h \}, \quad
h=120193920,
\]
and $\ell_{i,j}=4$ for all $i,j$.
We used PARI/GP to verify that
$4n_i/(n_j-n_i)\in \cV$ and $4n_j/(n_j-n_i)\in \cV$ for all $i<j$.
\end{proof}

%%%%%%%%%%%%%%%%%%%%%%%%%%%%%%%%%%%%%%%%%%%%%%%%%%%%
%
%
\section{A heuristic argument}\label{sec:heuristic}

In this section, we give an argument that
there should exist $m_1,\ldots,m_{50}$ satisfying
the hypothesis of Theorem \ref{DHLk} with $\ell_{i,j}=2$
for all $i,j$.  We first give a general construction
of numbers with $\frac{n_i}{n_j-n_i}$ all integers.

\begin{lem}\label{sets-b}
For any positive integer $b$  and and $k\ge 2$
there is a set $\{n_1,\ldots,n_k\}$ of positive integers
with $n_1<n_2<\cdots <n_k$ and with
\be\label{b-divis}
b \Big| \frac{n_j}{n_j-n_i} \quad (1\le i<j\le k).
\ee
\end{lem}

\begin{proof}
Induction on $k$.  When $k=2$ take $\{2b-1,2b\}$.
Now assume \eqref{b-divis} holds for some $k\ge 2$.  
Let $M$ be the least common
multiple of the $\binom{k}{2}$ numbers
\[
n_j-n_i \quad (1\le i<j\le k),
\]
and let $K$ be the least common multiple of the numbers
\[
M, bM-n_1,\ldots,bM-n_k.
\]
We claim that the set 
\[
\{n_1',\ldots,n_{k+1}'\} = \{Kb-bM+n_1,
 \ldots, Kb - bM +n_k, Kb \}
\]
satisfies \eqref{b-divis}.
Indeed, when $1\le i<j\le k$ we have
\[
\frac{n_{j}'}{n_{j}'-n_{i}'} = \frac{Kb - bM + n_j}{n_j-n_i},
\]
which by hypothesis is divisible by $b$.
Finally, for any $i\le k$,
\[
\frac{n_{k+1}'}{n_{k+1}'-n_i'} = \frac{Kb}{bM-n_i},
\]
which is also divisible by $b$.
\end{proof}

Now let $b=\prod_{p\le 2450} p$, and apply Lemma \ref{b-divis}
with $k=50$.  There is a set $\{n_1,\ldots,n_{50} \}$
such that for all $i<j$,
\be\label{b-divis-50}
b\, \Big| \, \frac{n_j}{n_j-n_i}.
\ee
Let $M$ be the least common multiple of
 the $\binom{50}{2}$ numbers
\[
n_j-n_i \quad (1\le i<j\le 50).
\]
Then for any $h\in \NN$, the set 
$\{n_1+hbM,\ldots,n_{50}+hbM\}$ has the same property.
The collection of 2450 linear forms (in $h$)
\[
\frac{2(n_i+hbM)}{n_{j}-n_i} + 1 = 
\frac{2(n_j+hbM)}{n_j-n_i} - 1 \qquad(1\le i < j\le 50)
\]
and
\[
\frac{2(n_j+hbM)}{n_j-n_i}+1\qquad
(1\le i < j\le 50)
\]
is admissible by \eqref{b-divis-50}, and
the Prime $k$-tuples conjecture implies
that all of these are prime for some $h$.
We need only the existence of one $h$,
and then the hypotheses of Theorem \ref{DHLk}
hold with $\ell_{i,j}=2$ for all $i,j$,
and consequently $4\in \cD$.  Discovering such an $h$,
however, appears to be computationally infeasible.

%%%%%%%%%%%%%%%%%%%%%%%%%%%%%%%%%%%%%%%%%%%%%%%%%%%%%%%%%%%%%%%%%%
%
%
\section{Totient gaps in progressions: proof of Theorem \ref{ap}}
\label{sec:ap}

\begin{lem}\label{ap-lem}
 Suppose that $D\in \NN$ and $4|a$.
Then there exist $v_1$ and $v_2$ such that $(D,v_1) = (D,v_2) = 1$ and
$(v_1-1)(v_2-1)\equiv a\pmod D$ or $(v_1+1)(v_2-1)\equiv a\pmod D$.
\end{lem}

\begin{proof}
We use the Chinese Remainder Theorem. We will prove that if $D=p^\alpha$
is a prime power, then for $p\neq3$ we can find a pair $v_1$ and $v_2$ such 
that $(D,v_1) = (D,v_2) = 1$ and $(v_1-1)(v_2-1)\equiv a\pmod D$
and another pair $v'_1$ and $v'_2$ such 
that $(D,v'_1) = (D,v'_2) = 1$ and $(v'_1+1)(v'_2-1)\equiv a\pmod D$.
If $p=3$ then we will find appropriate $v_1$ and $v_2$ such that
one of desired congruences hold. This will suffice for the proof of the 
lemma. 

%KF - reworded
If$p\neq 3$ and for any $a$, $4|a$, a pair
$v_1$ and $v_2$ exists such 
that $(D,v_1) = (D,v_2) = 1$ and $(v_1-1)(v_2-1)\equiv a\pmod D$, 
then there also exists a pair $v'_1$ and $v'_2$ as well. Indeed,
take any possible $D$ and $a$. Then, by our supposition, there are
$v_1$ and $v_2$ such 
that $(D,v_1) = (D,v_2) = 1$ and $(v_1-1)(v_2-1)\equiv -a\pmod D$.
Then for $v_1'=-v_1$ and $v_2'=v_2$ the desired congruence 
$(v'_1+1)(v'_2-1)\equiv a\pmod D$ holds.
  
Consider $p=2$.  Take $v_1\equiv 3\pmod{p^\a}$ and $v_2\equiv\frac{a}{2}+1 \mod{p^\a}$.
Then $(v_1-1)(v_2-1)\equiv a\pmod{p^\a}$ and both $v_1,v_2$ are odd.

Consider $p=3$. If $a\equiv0\pmod 3$ or $a\equiv1\pmod 3$, let
$v_1\equiv 2\pmod{3^\alpha}$ and $v_2\equiv a+1\pmod{3^\alpha}$. Then we have
$(v_1-1)(v_2-1)\equiv a\bmod 3^\alpha$. If $a\equiv-1\pmod 3$ let
$v_1\equiv -2\pmod{3^\alpha}$ and $v_2\equiv -a+1\pmod{3^\alpha}$.
Then we have $(v_1+1)(v_2-1)\equiv a\pmod{3^\alpha}$.

Consider $p>3$. Take $v_2$ so that $v_2\not\in\{0, 1, 1-a\}\mod p$.
Then there is some $v_1\not\equiv 0\pmod{p}$ such that
$(v_2-1)(v_1-1)\equiv a\pmod{p^\a}$.
\end{proof}

\begin{proof}[Proof of Theorem \ref{ap}]
Let $D$ be any positive integer satisfying
\begin{itemize}
\item[(a)] $d|D$;
\item[(b)] $D,D^2,\ldots,D^{49}$ are all in $\cV$.
\end{itemize}
For example, let $P$ be the largest prime factor of $d$,
$\gamma$ sufficiently large and
$$D = d\prod_{p\le P} p^\gamma.$$
Indeed, if
\[
D=\prod_{p\le P} p^{\alpha(p)}, \qquad \prod_{p\le P} (p-1) = \prod_{p\le P} p^{\beta(p)},
\]
then, assuming $\gamma \ge \max \beta(p)$, for all $j\ge 1$ we have
\[
\phi \bigg( \prod_{p\le P} p^{j\a(p)-\beta(p)+1} \bigg) = D^j.
\]

Now take any $D$ satisfying (a) and (b) above, and let
 $v$ be coprime to $D$.  Then
the set
%SK $$f_j(x) = D^jx + v,\quad j=1,\dots,50,$$
$$f_j(x) = D^jx - v,\quad j=1,\dots,50,$$
of linear forms is 
 admissible. Indeed,  if $p\nmid v$, then 
%SK $f_1(0)\cdots f_{50}(0) = v^k\not\equiv 0\pmod{p}$,
$f_1(0)\cdots f_{50}(0) = v^{50}\not\equiv 0\pmod{p}$,
and if $p|v$ then $p\nmid D$ and
$f_1(1)\cdots f_k(1) \equiv D^{1225}\not\equiv 0\pmod{p}$.
Since $\DHL[50,2]$ holds, 
there are $j_1<j_2$ such that for infinitely many positive
%SK integers $x$ both numbers $p_1=D^{j_1}x + v$, $p_2=D^{j_2}x + v$ are
integers $x$ both numbers $p_1=D^{j_1}x - v$, $p_2=D^{j_2}x - v$ are
primes. Denote $j=j_2-j_1$. There exists $l$ such that $\phi(l) = D^j$.
If $x$ is large enough, then $(p_1,l) = (p_2,l) = 1$. We have
\be\label{phi-p1lp2}
%SK \phi(p_1 l) - \phi(p_2) = (p_1-1)D^j - (p_2-1) = (v+1)(D^j-1).
\phi(p_2) - \phi(p_1 l)  = (p_2-1) - (p_1-1)D^j  = (v+1)(D^j-1).
\ee

Let $v_1,v_2$ be as in Lemma \ref{ap-lem},
and let $v$ satisfy 
\[
\begin{cases}
v\equiv -v_1 \pmod{D}, v>0 & \text{ if } (v_1-1)(v_2-1)\equiv a\pmod{D},\\
v\equiv v_1 \pmod{D}, v<-1 & \text{ otherwise}.
\end{cases}
\]

Fix a prime $q\equiv v_2\bmod D$ with $(q,l)=1$,
and assume that $p_1,p_2 > q$.
 Then
 \be\label{lpq}
%SK \phi(p_1 l q) - \phi(p_2 q) =(q-1)(v+1)(D^j-1).
\phi(p_2 q) - \phi(p_1 l q)  =(q-1)(v+1)(D^j-1).
\ee
Thus, $|(q-1)(v+1)(D^j-1)| \in \cD$.
The right side of \eqref{lpq} is $\equiv (q-1)(-v-1)\equiv (v_2-1)(-v-1)\pmod{D}$
and has sign equal to the sign of $v$.
If $(v_1-1)(v_2-1)\equiv a\pmod{D}$, then $v>0$
and thus the right side of \eqref{lpq} is 
positive and congruent to $(v_2-1)(v_1-1)\equiv a\pmod{D}$.
Otherwise, $v<0$ and the right side of \eqref{lpq} is 
negative and congruent to $(v_2-1)(-v_1-1)\equiv -a\pmod{D}$.
By varying $q$, we find that
there are infinitely many elements of $\cD$ 
in the residue class $a\bmod d$.
\end{proof}

\textbf{Remarks.}  Equation \eqref{phi-p1lp2} 
holds for any $v$ coprime to $D$.
Thus,
if $a\equiv 2\pmod{4}$ and either $(a+1,D)=1$ or $(a-1,D)=1$ then
the residue class $a\mod d$ contains infinitely many elements
of $\cD$; take $v\equiv -a-1\pmod{d}, v>0$ if $(a+1,D)=1$
and $v\equiv a-1\pmod{d}, v<-1$ if  $(a-1,D)=1$.
Thus, if $d$ has at most two distinct prime factors,
and (b) holds for some $D$ composed only of the primes 
dividing $d$, then
every residue class $a\mod d$, with $2|a$ contains
infinitely many elements of $\cD$.
Note that in this case, for all $a$ with $2|a$,
either $(a+1,d)=1$ or $(a-1,d)=1$.
In particular this holds with
$d$ of the form $2^k$, $2^k 3^\ell$, or $2^k 5^\ell$
with $k\ge 2$, since in each case (a) and (b)
hold with $D=d$.  Item (b) also holds with $d=D=28$
(verified with PARI/GP).  We do not know how to derive the
same conclusion if $d$ has 3 or more prime factors, e.g. $d=60$.

%%%%%%%%%%%%%%%%%%%%%%%%%

\end{document}